\documentclass[11pt]{amsart}
\usepackage[margin=1in]{geometry}
\usepackage{amsmath,amsthm,amssymb,mathtools}
\usepackage{microtype}
\usepackage[
  pdftitle={A Lorentz endpoint and quantitative obstruction for a Cowling--Price moment reduction},
  pdfauthor={Dean Menezes},
  pdfsubject={Uncertainty principles, weighted moment inequalities, and Lorentz spaces},
  pdfkeywords={
    uncertainty principle,
    Cowling--Price inequality,
    Lorentz spaces,
    Carlson inequality,
    weighted moment inequality,
    Fourier transform
  },
  hidelinks
]{hyperref}
\newtheorem{theorem}{Theorem}[section]
\newtheorem{proposition}[theorem]{Proposition}

\newtheorem{corollary}[theorem]{Corollary}
\theoremstyle{remark}
\newtheorem{remark}[theorem]{Remark}

\newcommand{\R}{\mathbb R}
\newcommand{\norm}[1]{\lVert #1\rVert}
\newcommand{\ind}{\mathbf 1}
\newcommand{\Mrp}{M_{r,p}^{(d)}}
\renewcommand{\Mc}{M_{d(p-1),p}^{(d)}}

\date{July 18, 2026}

\hypersetup{
  pdftitle={A Lorentz endpoint and quantitative obstruction for a Cowling--Price moment reduction},
  pdfauthor={Dean Menezes}
}

\title{A Lorentz endpoint and quantitative obstruction
for a Cowling--Price moment reduction}

\author{Dean Menezes}
\address{
University of Texas at Austin,
Austin, Texas, USA
}
\email{dean.menezes@utexas.edu}



\begin{document}
\maketitle

\begin{abstract}
Let
\[
 M_{r,p}^{(d)}(g)=\int_{\R^d}|x|^r|g(x)|^p\,dx,
 \qquad 1<p<\infty.
\]
Write $p'=p/(p-1)$ and $\omega_d=|B(0,1)|$.  A direct step in the
Wigderson--Wigderson moment method asks for a
single-function estimate
\[
 \norm{g}_1\le K\norm{g}_q^{\,1-p\eta}
 \bigl(M_{r,p}^{(d)}(g)\bigr)^\eta,
 \qquad
 \eta=\frac{d(1-1/q)}{r+d-dp/q}.
\]
We give a self-contained proof that, in its scale-invariant regime, this
estimate holds exactly when $r>d(p-1)$.  We then quantify the transition:
the best constant blows up like
$[r-d(p-1)]^{-1/p'}$ at the critical line, while on the annulus
$1\le |x|\le R$ the endpoint loss is exactly
$(d\omega_d\log R)^{1/p'}$.  The critical line nevertheless has two
sharp Lorentz endpoints.  The strong moment controls $L^{1,p}$,
\[
 \norm{g}_{L^{1,p}}\le \omega_d^{1/p'}\Mc(g)^{1/p},
\]
and strengthening the weighted norm from $L^p$ to $L^{p,1}$ restores the
strong target,
\[
 \norm{g}_1\le \omega_d^{1/p'}
 \norm{|x|^{d/p'}g}_{L^{p,1}}.
\]
More generally, every power moment has a sharp target Lorentz space, and
its secondary index is optimal.  These results identify the endpoint lost
by the direct same-$q$ reduction.  They do not obstruct auxiliary-norm
versions of the Wigderson framework, which reach the Cowling--Price range
by a different route.
\end{abstract}

\section{Introduction}

For $r>0$, $1<p<\infty$, and a measurable function $g$ on $\R^d$, set
\begin{equation}\label{eq:moment}
 \Mrp(g)=\int_{\R^d}|x|^r|g(x)|^p\,dx.
\end{equation}
Cowling and Price~\cite{CP84} proved broad uncertainty principles asserting
that a function and its Fourier transform cannot both have small weighted
norms.  Wigderson and Wigderson~\cite{WW21} later developed a soft method
that derives many uncertainty inequalities from a primary norm uncertainty
principle.

One version of their argument starts from an estimate of the form
\begin{equation}\label{eq:normUP}
 \frac{\norm{f}_1}{\norm{f}_q}
 \frac{\norm{Af}_1}{\norm{Af}_q}\ge \kappa,
 \qquad 1<q\le\infty,
\end{equation}
for an operator $A$.  To convert \eqref{eq:normUP} into a moment inequality,
one would like to control each $L^1/L^q$ ratio separately by a weighted
moment.  Dilation forces the comparison to have the form
\begin{equation}\label{eq:direct}
 \norm{g}_1
 \le K_{d,r,p,q}\norm{g}_q^{\,1-p\eta}\Mrp(g)^\eta,
 \qquad
 \eta=\frac{d(1-1/q)}{D},
 \qquad
 D=r+d-\frac{dp}{q}.
\end{equation}
Here and below $1/q=0$ when $q=\infty$.

The usual proof splits the $L^1$ norm at a radius $T$.  Its tail contains
\[
 \int_{|x|>T}|x|^{-r/(p-1)}\,dx,
\]
which converges exactly when
\begin{equation}\label{eq:critical-line}
 r>d(p-1).
\end{equation}
In dimension one this is the threshold $r>p-1$ appearing in the original
manuscript.  The first purpose of this paper is to describe what happens at
and near the line \eqref{eq:critical-line}, rather than merely record the
failure below it.

The strong inequality \eqref{eq:direct} is a homogeneous weighted Carlson
inequality.  General sharp results, including best constants and equality
cases, were obtained by Barza, Burenkov, Pe\v{c}ari\'{c}, and
Persson~\cite{BBPP98}; see also Larsson~\cite{Larsson03}.  More recent
extensions treat several homogeneous weights and optimal-recovery
problems~\cite{Osipenko24}.  Thus the threshold for the strong inequality
should be viewed as a specialization of the Carlson theory.

What is useful in the uncertainty-principle setting is the endpoint
structure.  We record an elementary rearrangement consequence showing that
the strong moment at the critical line lands sharply in $L^{1,p}$ rather
than $L^1$, and that the complementary source refinement $L^{p,1}$ restores
the $L^1$ conclusion.  We also obtain the exact logarithmic loss on finite
annuli and the order of blow-up of the best supercritical constant.  Since the embedding is an elementary consequence of rearrangement, we
make no priority claim for the underlying statement; our point is its
endpoint interpretation and the accompanying quantitative transition.

The scope of the obstruction is deliberately narrow.  Tang~\cite{Tang25}
and Costa Dias, Luef, and Prata~\cite{DLP25} use more flexible primary norm
principles and auxiliary exponents to reach Cowling--Price-type results
without passing through \eqref{eq:direct}.  Huang, Li, and Liu~\cite{HLL}
have since developed an abstract $L^p$ framework that further unifies this
approach.  Accordingly, our conclusion concerns the direct same-$q$
single-function reduction, not the full Wigderson strategy.  Related
rearrangement-invariant uncertainty inequalities and isoperimetric weights
were developed by Mart\'in and Milman~\cite{MartinMilman16}; Lorentz-space
uncertainty principles in a different direction were studied by Fu and
Xiao~\cite{FuXiao23}.

Section~\ref{sec:lorentz} proves the sharp moment-to-Lorentz embedding and
its optimality.  Section~\ref{sec:direct} establishes the direct strong
estimate in $\R^d$.  Section~\ref{sec:quantitative} gives the critical
logarithmic loss and the near-critical blow-up.  Section~\ref{sec:holder}
examines a same-$q$ three-factor H\"older refinement, and
Section~\ref{sec:scope} records the operator endpoint and the precise scope
of the result.

\section{The sharp Lorentz target of a power moment}\label{sec:lorentz}

Let $h^*$ denote the nonincreasing rearrangement of $|h|$.  For
$0<s<\infty$ and $0<\tau<\infty$, we use the Lorentz functional
\begin{equation}\label{eq:lorentz}
 \norm{h}_{L^{s,\tau}}
 =\left(\int_0^\infty
 [t^{1/s}h^*(t)]^\tau\,\frac{dt}{t}\right)^{1/\tau},
\end{equation}
and
\[
 \norm{h}_{L^{s,\infty}}=\sup_{t>0}t^{1/s}h^*(t).
\]
With this normalization $L^{s,s}=L^s$ isometrically.  We write
$\omega_d=|B(0,1)|$, so $|S^{d-1}|=d\omega_d$.

The following theorem identifies the sharp Lorentz target generated by
one power moment.

\begin{theorem}[Sharp moment-to-Lorentz embedding]\label{thm:moment-lorentz}
Let $d\ge1$, $1<p<\infty$, and $r>0$, and put
\begin{equation}\label{eq:sr}
 s_r=\frac{dp}{d+r}.
\end{equation}
Then every measurable $g$ satisfies
\begin{equation}\label{eq:moment-lorentz}
 \norm{g}_{L^{s_r,p}(\R^d)}
 \le \omega_d^{\,r/(dp)}\Mrp(g)^{1/p}.
\end{equation}
The constant is sharp.  Equality holds whenever $|g|$ is radially
symmetric and nonincreasing about the origin.
\end{theorem}

\begin{proof}
Let $g^\#$ be the symmetric decreasing rearrangement of $|g|$.  Since
$|x|^r$ is radially increasing, the layer-cake representation and the fact
that a centered ball minimizes the integral of $|x|^r$ among sets of fixed
measure give
\begin{equation}\label{eq:bathtub}
 \int_{\R^d}|x|^r|g(x)|^p\,dx
 \ge \int_{\R^d}|x|^r g^\#(x)^p\,dx.
\end{equation}
This is the standard bathtub or rearrangement principle; see, for example,
\cite{LiebLoss}.

Because $g^\#(x)=g^*(\omega_d|x|^d)$, polar coordinates followed by the
substitution $t=\omega_d\rho^d$ give
\begin{align*}
 \int_{\R^d}|x|^r g^\#(x)^p\,dx
 &=d\omega_d\int_0^\infty
   \rho^{d+r-1}g^*(\omega_d\rho^d)^p\,d\rho\\
 &=\omega_d^{-r/d}\int_0^\infty t^{r/d}g^*(t)^p\,dt.
\end{align*}
Since $p/s_r-1=r/d$, the last integral is
$\omega_d^{-r/d}\norm{g}_{L^{s_r,p}}^p$.  Combining this identity with
\eqref{eq:bathtub} proves \eqref{eq:moment-lorentz}.  If $|g|$ is radially
nonincreasing, equality holds in \eqref{eq:bathtub}; hence the constant is
attained and is sharp.
\end{proof}

The secondary Lorentz index in Theorem~\ref{thm:moment-lorentz} cannot be
improved.

\begin{proposition}[Optimal secondary index]\label{prop:secondary}
Fix $d,p,r$ as above, let $s_r$ be given by \eqref{eq:sr}, and let
$0<\tau\le\infty$.  There is a finite constant $C$ such that
\begin{equation}\label{eq:secondary}
 \norm{g}_{L^{s_r,\tau}}
 \le C\Mrp(g)^{1/p}
\end{equation}
for every measurable $g$ if and only if $\tau\ge p$.
\end{proposition}

\begin{proof}
For $\tau\ge p$, the conclusion follows from
Theorem~\ref{thm:moment-lorentz} and the standard Lorentz inclusion
$L^{s_r,p}\hookrightarrow L^{s_r,\tau}$; see
\cite{BennettSharpley}.

Suppose $0<\tau<p$.  For $R>1$, let $g_R$ be symmetric decreasing with
rearrangement
\[
 g_R^*(t)=
 \begin{cases}
  1,&0<t<1,\\
  t^{-1/s_r},&1\le t\le R,\\
  0,&t>R.
 \end{cases}
\]
Then
\[
 \norm{g_R}_{L^{s_r,p}}^p
 =\frac{s_r}{p}+\log R,
 \qquad
 \norm{g_R}_{L^{s_r,\tau}}^\tau
 =\frac{s_r}{\tau}+\log R.
\]
Equality holds in Theorem~\ref{thm:moment-lorentz} for $g_R$, so
$\Mrp(g_R)^{1/p}\asymp(\log R)^{1/p}$.  Therefore the quotient in
\eqref{eq:secondary} grows like
$(\log R)^{1/\tau-1/p}$ and tends to infinity.
\end{proof}

At the critical line $r=d(p-1)$, one has $s_r=1$.

\begin{corollary}[Critical target endpoint]\label{cor:target-endpoint}
For $1<p<\infty$,
\begin{equation}\label{eq:target-endpoint}
 \norm{g}_{L^{1,p}(\R^d)}
 \le \omega_d^{1/p'}\Mc(g)^{1/p}.
\end{equation}
The constant is sharp, and the target secondary index $p$ is optimal.
In particular, the critical strong moment does not control $L^1$.
\end{corollary}

The failure of $L^1$ can also be repaired on the source side by replacing
the weighted $L^p$ norm with the strictly stronger Lorentz norm $L^{p,1}$.
This is the critical pairing $L^{p,1}\cdot L^{p',\infty}\subset L^1$ from
Lorentz-space H\"older theory; compare O'Neil~\cite{ONeil63}.  In the
present radial setting, rearrangement also gives the exact constant.

\begin{theorem}[Complementary source endpoint]\label{thm:source-endpoint}
For $1<p<\infty$,
\begin{equation}\label{eq:source-endpoint}
 \norm{g}_{L^1(\R^d)}
 \le \omega_d^{1/p'}
 \norm{|x|^{d/p'}g}_{L^{p,1}(\R^d)}.
\end{equation}
The constant is sharp.  Equality is attained by
$g(x)=|x|^{-d/p'}\ind_{B_R}(x)$ for every $R>0$.
\end{theorem}

\begin{proof}
Set $h(x)=|x|^{d/p'}|g(x)|$ and $w(x)=|x|^{-d/p'}$.  The rearrangement of
$w$ is
\[
 w^*(t)=\left(\frac{\omega_d}{t}\right)^{1/p'}.
\]
The Hardy--Littlewood rearrangement inequality therefore gives
\begin{align*}
 \norm{g}_1
 &=\int_{\R^d}h(x)w(x)\,dx\\
 &\le\int_0^\infty h^*(t)w^*(t)\,dt\\
 &=\omega_d^{1/p'}\int_0^\infty
    t^{1/p}h^*(t)\,\frac{dt}{t},
\end{align*}
which is \eqref{eq:source-endpoint}.  If
$h=\ind_{B_R}$, then equality holds throughout.  This corresponds to the
stated function $g$.
\end{proof}

Theorems~\ref{thm:moment-lorentz} and~\ref{thm:source-endpoint} are two
complementary endpoint statements.  Keeping the strong weighted $L^p$
moment forces a weaker target $L^{1,p}$; keeping the strong target $L^1$
forces the finer source space $L^{p,1}$.  Both constants are
$\omega_d^{1/p'}$ at criticality.

\section[The direct same-q estimate]{The direct same-$q$ estimate}\label{sec:direct}

We now return to \eqref{eq:direct}.  The exponent $\eta$ is forced by the
two homogeneities of the problem.  Indeed, for
$h(x)=c\,g(x/\lambda)$,
\[
 \norm{h}_1=c\lambda^d\norm{g}_1,
 \qquad
 \norm{h}_q=c\lambda^{d/q}\norm{g}_q,
 \qquad
 \Mrp(h)=c^p\lambda^{r+d}\Mrp(g).
\]
Thus \eqref{eq:direct} is invariant precisely for the exponent displayed
there.

\begin{theorem}[Sharp threshold]\label{thm:direct}
Let $d\ge1$, $1<p<\infty$, $1<q\le\infty$, and $r>0$.  Assume
\begin{equation}\label{eq:Dpositive}
 D=r+d-\frac{dp}{q}>0,
\end{equation}
and define $\eta=d(1-1/q)/D$.  Then a finite constant in
\eqref{eq:direct} exists if and only if
\[
 r>d(p-1).
\]
For every $r\le d(p-1)$ satisfying \eqref{eq:Dpositive}, the
scale-invariant quotient in \eqref{eq:direct} is unbounded.  The threshold
is independent of $q$.
\end{theorem}

\begin{proof}
Suppose first that $r>d(p-1)$.  Put
\[
 A=d\left(1-\frac1q\right),
 \qquad
 B=\frac{r-d(p-1)}{p}.
\]
Both are positive.  For $T>0$, H\"older's inequality on the ball gives
\begin{equation}\label{eq:inside}
 \int_{|x|\le T}|g(x)|\,dx
 \le \omega_d^{1-1/q}T^A\norm{g}_q.
\end{equation}
For the tail,
\begin{align}
 \int_{|x|>T}|g(x)|\,dx
 &\le
 \left(\int_{|x|>T}|x|^{-r/(p-1)}\,dx\right)^{1/p'}
 \Mrp(g)^{1/p}\notag\\
 &\le C_{d,r,p}T^{-B}\Mrp(g)^{1/p}.
 \label{eq:tail}
\end{align}
The pure-power integral is finite exactly because $r>d(p-1)$.  Optimizing
the sum of \eqref{eq:inside} and \eqref{eq:tail} yields
\[
 \norm{g}_1
 \le C\norm{g}_q^{\,B/(A+B)}
       \Mrp(g)^{\,A/[p(A+B)]}.
\]
Since
\[
 p(A+B)=D,
 \qquad
 \frac{A}{p(A+B)}=\eta,
 \qquad
 \frac{B}{A+B}=1-p\eta,
\]
this is \eqref{eq:direct}.

Now suppose $r<d(p-1)$.  Choose
\begin{equation}\label{eq:alpha-choice}
 \max\left\{\frac dq,\frac{r+d}{p}\right\}<\alpha<d
\end{equation}
(with $d/q=0$ for $q=\infty$), and set
\[
 g_R(x)=|x|^{-\alpha}\ind_{\{1\le |x|\le R\}}.
\]
Then
\[
 \norm{g_R}_1\asymp R^{d-\alpha}\longrightarrow\infty,
\]
whereas \eqref{eq:alpha-choice} implies that $\norm{g_R}_q$ and
$\Mrp(g_R)$ remain bounded above and below by positive constants as
$R\to\infty$.  Hence the quotient in \eqref{eq:direct} diverges.

At the endpoint $r=d(p-1)$, take
\begin{equation}\label{eq:endpoint-power}
 g_R(x)=|x|^{-d}\ind_{\{1\le |x|\le R\}}.
\end{equation}
Then $\norm{g_R}_q$ remains bounded above and below, while
\begin{equation}\label{eq:endpoint-logs}
 \norm{g_R}_1=d\omega_d\log R,
 \qquad
 \Mc(g_R)=d\omega_d\log R.
\end{equation}
At this line $\eta=1/p$ and $1-p\eta=0$, so the scale-invariant quotient
grows like $(\log R)^{1/p'}$.  This proves failure at criticality.
\end{proof}

\begin{remark}[The regime $D\le0$]\label{rem:D}
The restriction \eqref{eq:Dpositive} is essential.  If $D=0$, no finite
exponent can make \eqref{eq:direct} dilation invariant.  If $D<0$, the
forced exponent $\eta$ is negative.  Translating a fixed nonzero compactly
supported function leaves its $L^1$ and $L^q$ norms unchanged while making
its moment tend to infinity, so the right-hand side of \eqref{eq:direct}
tends to zero.
\end{remark}

\section{Quantitative behavior at criticality}\label{sec:quantitative}

The preceding proof gives more than a qualitative threshold.  Let
$K_{d,r,p,q}$ denote the best constant in \eqref{eq:direct} for
$r>d(p-1)$.

\begin{theorem}[Blow-up of the best constant]\label{thm:blowup}
Fix $d\ge1$, $1<p<\infty$, and $1<q\le\infty$.  As
\[
 \delta:=r-d(p-1)\downarrow0,
\]
one has
\begin{equation}\label{eq:blowup}
 K_{d,r,p,q}\asymp \delta^{-1/p'}.
\end{equation}
The implicit constants may depend on $d,p,q$ but not on sufficiently small
$\delta>0$.
\end{theorem}

\begin{proof}
Put
\[
 A=d\left(1-\frac1q\right),
 \qquad B=\frac{\delta}{p},
 \qquad
 \eta_\delta=\frac{A}{pA+\delta}.
\]
The tail integral in \eqref{eq:tail} can now be evaluated exactly:
\begin{align*}
 \left(\int_{|x|>T}|x|^{-r/(p-1)}\,dx\right)^{1/p'}
 &=\left(\frac{d\omega_d(p-1)}{\delta}\right)^{1/p'}
 T^{-\delta/p}.
\end{align*}
Consequently
\begin{equation}\label{eq:quant-split}
 \norm{g}_1
 \le c_0T^A\norm{g}_q
 +c_1\delta^{-1/p'}T^{-B}\Mrp(g)^{1/p},
\end{equation}
where $c_0,c_1$ are independent of small $\delta$.

For $u,v>0$,
\begin{equation}\label{eq:optimization}
 \inf_{T>0}(uT^A+vT^{-B})
 =C_{A,B}u^{B/(A+B)}v^{A/(A+B)},
\end{equation}
where
\[
 C_{A,B}=(A+B)A^{-A/(A+B)}B^{-B/(A+B)}.
\]
The factor $C_{A,B}$ stays bounded as $B\downarrow0$.  Applying
\eqref{eq:optimization} to \eqref{eq:quant-split}, and noting that
$A/(A+B)=p\eta_\delta<1$, gives
\[
 K_{d,r,p,q}
 \le C\delta^{-p\eta_\delta/p'}
 \le C\delta^{-1/p'}
\]
for $0<\delta<1$.

For the lower bound, let
\[
 R_\delta=e^{1/\delta},
 \qquad
 g_\delta(x)=|x|^{-d}\ind_{\{1\le |x|\le R_\delta\}}.
\]
Then
\[
 \norm{g_\delta}_1=\frac{d\omega_d}{\delta},
 \qquad
 \Mrp(g_\delta)=\frac{d\omega_d(e-1)}{\delta}.
\]
Moreover, $\norm{g_\delta}_q$ stays bounded above and below by positive
constants; for $q=\infty$ it equals $1$.  Hence the quotient defining the
best constant is bounded below by a constant multiple of
\[
 \delta^{-1+\eta_\delta}
 =\delta^{-1/p'}\delta^{\eta_\delta-1/p}.
\]
Since
\[
 \eta_\delta-\frac1p
 =-\frac{\delta}{p(pA+\delta)},
\]
the last factor stays bounded below by a positive constant as
$\delta\downarrow0$.  This proves \eqref{eq:blowup}.
\end{proof}

At the critical line, the failure on a finite annulus has an exact constant.

\begin{proposition}[Exact annular logarithmic loss]\label{prop:annulus}
For $R>1$, define
\[
 \mathcal A_R=
 \sup_{\substack{g\ne0\\
 \operatorname{supp}g\subset\{1\le |x|\le R\}}}
 \frac{\norm{g}_1}{\Mc(g)^{1/p}}.
\]
Then
\begin{equation}\label{eq:annulus}
 \mathcal A_R=(d\omega_d\log R)^{1/p'}.
\end{equation}
Equality is attained by every nonzero scalar multiple of
$|x|^{-d}\ind_{\{1\le |x|\le R\}}$.
\end{proposition}

\begin{proof}
H\"older's inequality gives
\begin{align*}
 \norm{g}_1
 &=\int_{1\le|x|\le R}
   \bigl(|x|^{d/p'}|g(x)|\bigr)|x|^{-d/p'}\,dx\\
 &\le \Mc(g)^{1/p}
 \left(\int_{1\le|x|\le R}|x|^{-d}\,dx\right)^{1/p'}\\
 &=\Mc(g)^{1/p}(d\omega_d\log R)^{1/p'}.
\end{align*}
The equality condition in H\"older's inequality is satisfied by the stated
power law, proving sharpness.
\end{proof}

Proposition~\ref{prop:annulus} is the finite-scale form of the endpoint
obstruction: the loss is not merely logarithmic up to constants, but exactly
the $L^{p'}$ norm of the critical inverse weight on the annulus.  The
Lorentz endpoint \eqref{eq:target-endpoint} removes this growing factor by
recording the distribution of the critical power law rather than only its
strong $L^p$ size.

\section[A same-q three-factor Holder refinement]{A same-$q$ three-factor H\"older refinement}\label{sec:holder}

A natural attempted improvement is to feed the available $L^q$ information
into the tail estimate itself.  For $0\le\theta\le1$, write
\[
 |g(x)|
 =\bigl(|x|^{-r/p}\bigr)^\theta
  \bigl(|x|^{r/p}|g(x)|\bigr)^\theta
  |g(x)|^{1-\theta}.
\]
Apply H\"older with
\[
 u_2=\frac p\theta,
 \qquad
 u_3=\frac q{1-\theta},
 \qquad
 \frac1{u_1}=1-\frac\theta p-\frac{1-\theta}{q},
\]
using the usual endpoint conventions.

\begin{proposition}\label{prop:holder}
The pure-power factor in this tail estimate is integrable on
$\{|x|>T\}$ precisely when
\begin{equation}\label{eq:holder-condition}
 \theta\left(r+d-\frac{dp}{q}\right)
 >dp\left(1-\frac1q\right).
\end{equation}
If any $\theta\in[0,1]$ satisfies \eqref{eq:holder-condition}, then
$\theta=1$ does.  At $\theta=1$, the condition is exactly
$r>d(p-1)$.  Thus this same-$q$ H\"older family does not improve the
strong-moment range.
\end{proposition}

\begin{proof}
The pure-power factor is $|x|^{-\theta r/p}$ in $L^{u_1}$, and its tail is
integrable if and only if
\[
 \frac{\theta r u_1}{p}>d.
\]
Substituting the value of $u_1$ and clearing denominators gives
\eqref{eq:holder-condition}.  If
$D=r+d-dp/q\le0$, the condition is impossible.  If $D>0$, its left side
is increasing in $\theta$, so it is easiest to satisfy at $\theta=1$.
There it reduces to
\[
 r+d-\frac{dp}{q}>dp-\frac{dp}{q},
\]
which is equivalent to $r>d(p-1)$.
\end{proof}

This proposition treats the displayed decomposition with the same $L^q$
norm.  It is not a classification of all auxiliary-norm arguments.

\section{Operator endpoint and scope}\label{sec:scope}

The source endpoint can be inserted directly into a primary norm uncertainty
principle.  Define
\[
 \Lambda_p(g)=\norm{|x|^{d/p'}g}_{L^{p,1}(\R^d)}.
\]

\begin{corollary}[Critical endpoint for the direct operator reduction]
\label{cor:operator}
Suppose an operator $A$ on functions on $\R^d$ satisfies
\eqref{eq:normUP}.  Then
\begin{equation}\label{eq:operator-endpoint}
 \frac{\Lambda_p(f)}{\norm{f}_q}
 \frac{\Lambda_p(Af)}{\norm{Af}_q}
 \ge \omega_d^{-2/p'}\kappa
\end{equation}
whenever the displayed quantities are finite and nonzero.
\end{corollary}

\begin{proof}
Theorem~\ref{thm:source-endpoint} gives
$\Lambda_p(g)\ge\omega_d^{-1/p'}\norm{g}_1$.  Apply this to $f$ and
$Af$, multiply, and use \eqref{eq:normUP}.
\end{proof}

Thus the direct method does possess a genuine endpoint, but its natural
critical localization functional is Lorentz rather than strong $L^p$.
With the original strong moment, Theorem~\ref{thm:direct} and
Proposition~\ref{prop:annulus} show exactly why the $L^1$ comparison fails.

These statements should not be read as an obstruction to the broader
Wigderson framework.  Tang~\cite{Tang25} changes the intermediate exponent
and uses Hausdorff--Young; Costa Dias, Luef, and Prata~\cite{DLP25} formulate
a more general primary uncertainty principle; and Huang, Li, and
Liu~\cite{HLL} place such arguments in an abstract $L^p$ setting.  Those
methods do not require the same-$q$ estimate \eqref{eq:direct} and therefore
are not constrained by its critical line.

A natural next problem is coupled rather than single-function: determine
whether Fourier information permits critical Lorentz indices weaker than
those obtainable from the one-sided embeddings alone.  Such a result would
have to use the relation between $f$ and $\widehat f$, not just apply the
same weighted comparison independently to the two functions.

\end{document}